\documentclass[11pt,english,twoside,a4paper]{amsart}     
\usepackage[english,frenchb]{babel}
\usepackage{amssymb,amsmath,amsthm,stmaryrd}
\usepackage{mathrsfs}
\usepackage{graphicx,psfrag,epsfig}
\usepackage{graphicx}
\usepackage[T1]{fontenc}
\usepackage[latin1]{inputenc}
\usepackage[all]{xy}
\usepackage{mathrsfs}
\usepackage{pst-all}
\usepackage{float}
\theoremstyle{plain}

\theoremstyle{plain}

\newtheorem{thm}{Theorem}
\newtheorem{prop}{Proposition}[section]
\newtheorem{lem}{Lemma}[section]

\newtheorem{propa}{Proposition}
\newtheorem{cor}{Corollary}[section]
\newtheorem{pte}{Property}[section]

\theoremstyle{definition}
\newtheorem{Def}{Definition}[section]

\theoremstyle{remark}
\newtheorem{rem}{Remark}[section]
\newtheoremstyle{restate}
{\topsep}
{\topsep}
{\itshape}
{}
{\bfseries}
{.}
{ }
{\thmname{#1}\thmnote{ #3}}
\theoremstyle{restate}

\newtheorem{thm*}{Theorem}
\newtheorem{prop*}{Proposition}
\newtheorem{lem*}{Lemma}
\newtheorem{add*}{Addendum}
\newtheorem{cor*}{Corollary}
\newtheorem{pte*}{Property}
\theoremstyle{definition}
\newtheorem{Def*}{Definition}
\newtheorem{exm*}{Example}
\theoremstyle{remark}
\newtheorem{rem*}{Remark}

\theoremstyle{definition}

\theoremstyle{remark}

\newcommand{\ov}{\overline}
\newcommand{\Id}{\mathop{\hbox{{\rm Id}}}\nolimits}

\newcommand{\diam}{{\rm diam\,}}

\newcommand{\Dev}{{\rm Dev\,}}
\newcommand{\ha}{\widehat}

\newcommand{\wti}{\widetilde}
\newcommand{\N}{\mathbb N}
\newcommand{\Z}{\mathbb Z}
\newcommand{\Q}{\mathbb Q}
\newcommand{\R}{\mathbb R}

\newcommand{\T}{\mathbb T}

\newcommand{\Cal}{\mathcal}

\newcommand{\mcr}{\mathscr}

\newcommand{\abs}[1]{\left\vert #1\right\vert}
\newcommand{\norm}[1]{\Vert#1\Vert}

\newcommand{\demi}{\frac{1}{2}}

\newcommand{\setm}{\setminus}

\DeclareMathOperator{\Per}{Per}
\DeclareMathOperator{\Fix}{Fix}
\DeclareMathOperator{\Hom}{Homeo}

\newcommand{\rk}{{\rm rank\,}}

\newcommand{\Bb}{\Cal B}

\newcommand{\Ii}{\Cal I}
\newcommand{\Jj}{\Cal J}

\newcommand{\CC}{\mcr C}

\newcommand{\BB}{\mcr B}

\newcommand{\PP}{\mcr P}
\newcommand{\NN}{\mcr N}
\newcommand{\II}{\mcr I}
\newcommand{\JJ}{\mcr J}

\newcommand{\OO}{\mcr O}

\newcommand{\al}{\alpha}
\newcommand{\be}{\beta}
\newcommand{\ga}{\gamma}
\newcommand{\sig}{\sigma}
\newcommand{\eps}{\varepsilon}
\newcommand{\Ga}{\Gamma}
\renewcommand{\th}{\theta}
\newcommand{\vp}{\varphi}

\newcommand{\Sig}{\Sigma}
\newcommand{\De}{\Delta}
\newcommand{\om}{\omega}
\newcommand{\Om}{\Omega}

\newcommand{\de}{\delta}

\newcommand{\ka}{\kappa}

\newcommand{\rit}{\rightarrow}
\newcommand{\ma}{\mapsto}

\newcommand{\ee}{\frac{\eps}{2}}

\newcommand{\inv}{^{-1}}

\newcommand{\bz}{\bar{z}}

\newcommand{\hp}{{\rm{h_{pol}}}}

\newcommand{\pq}{\frac{p}{q}}

\begin{document}
\selectlanguage{english}

\author{Clémence Labrousse}
\title[circle homeomorphisms  and $C^1$ nonvanishing vector fields on $\T^2$]{Polynomial entropy for the circle homeomorphisms  and for $C^1$ nonvanishing vector fields on $\T^2$}
\thanks{CEREMADE, Université Paris-Dauphine,
Place du maréchal de Lattre de Tassigny
email: labrousse@ceremade.dauphine.fr}

\date{}

\maketitle

\begin{abstract} 
We prove that the polynomial entropy of an orientation preserving homeomorphism of the circle equals $1$ when the homeomorphism is not conjugate to a rotation and that it is $0$ otherwise.
In a second part we prove that the polynomial entropy  of a flow on the two dimensional torus associated with a $C^1$ nonvanishing vector field is less that $1$. We moreover prove that when the flow possesses periodic orbits its polynomial entropy equals $1$ unless it is conjugate to a rotation (in this last case, the polynomial entropy is zero). 
\end{abstract}


\section{Introduction}
It is a classical althought blurred question to estimate the complexity of a dynamical system. 
One tool to make this more precise is the \emph{topological entropy} which may be seen as the exponential growth rate of the number of orbits of the system one needs to know in order to understand the whole set of orbits within a given precision. 
One often distinguishes between systems with zero topological entropy and systems with positive topological entropy and there exists several criterions to prove that a given system has a positive topological entropy.
Even if there are no criterion for a system to have zero entropy (unless being an isometry or contracting), there are several well-known zero entropy continuous systems: for instance the harmonic and anharmonic oscillators, the simple pendulum, orientation preserving homeomorphisms of the circle,  elliptic billiards...
Although all these systems have the same topological entropy, it seems obvious that the harmonic oscillator is simpler than the anharmonic one which is simpler than the simple pendulum. In the same way,  a rotation on the circle is simpler than a homeomorphism that possesses both periodic and wandering points or Denjoy sets and finally circular billiard looks simpler than any other elliptic billiards. 
It is therefore a natural question to estimate the complexity of such systems more precisely.
To do this, we no longer use an \emph{exponential} measure of the complexity, like the topological entropy, but a \emph{polynomial} measure of the complexity, namely the polynomial entropy. 
In the last years, the polynomial entropy has been studied by J-P Marco in the framework of Liouville integrable Hamiltonian systems. In \cite{Mar-09}, he computed the polynomial entropy for action-angle systems  and Hamiltonian systems on surfaces defined by a Morse function. 
As an application of these results one sees that the polynomial entropy of the harmonic oscillator is smaller than that of the anharmonic oscillator which is smaller than that of the simple pendulum.
But it turns out that the computation of the polynomial entropy is in general an intricate problem even for systems in low dimension. 

In the present paper, we compute the polynomial entropy $\hp$ for orientation preserving homeomorphisms of the circle and give estimates for that of nonvanishing vector fields on the torus. Then main results are the following.

\begin{thm}\label{cercle}
Let $f\in\Hom^+(\T)$. Then $\hp(f)\in\{0,1\}$ and $\hp(f)=0$ if and only if $f$ is conjugate to a rotation.
\end{thm}

\begin{thm}\label{tore}
Let $X$ be a $C^1$ nonvanishing vector field on the torus $\T^2$ with associated flow $\phi_X=(\phi_X^t)_{t\in \R}$.
Then $\hp(\phi_X)\in [0,1]$.
Moreover, if $\phi_X$ possesses periodic orbits, $\hp(\phi_X)\in\{0,1\}$ and $\hp(\phi_X)=0$ if and only if $\phi_X^1$ is conjugate to a rotation.
\end{thm}

The organisation of the paper is the following. In Section 2 we recall the definition and some classical properties of the polynomial entropy.
In Section~3, we prove Theorem \ref{cercle} and in Section 4, we prove Theorem \ref{tore}.

\subsection*{Acknowledgements} Part of this paper was realized when I was in a postdoctoral position at the Université de Neuchâtel and when I was invited at ETH Zürich. I am grateful to Felix Schlenk for the postdoctoral position in Neuchâtel and to Paul Biran for the invitation at ETH Zürich. 

Finally, I wish to thank warmly Jean-Pierre Marco for numerous helpful discussions, notably on the estimates of the polynomial entropy of $C^1$ nonvanishing vector field on the torus.

\section{The polynomial entropy}

We consider a continuous  map $f:X\rit X$, where $(X,d)$ is a compact metric space $(X,d)$. 
Let $(X,d)$ be a compact metric space and $f$ a continuous map $X\rit X$. We construct new metrics $d_n^f$ on $X$, which depend on the iterations of $f$, by setting 
\[
d_n^f(x,y)=\max_{0\leq k\leq n-1}d(f^k(x),f^k(y)).
\]
These metrics are the \textit{dynamical metrics} associated with $f$. Obviously, if $f$ is an isometry or is contracting, $d^f_n$ coincides with $d$ and in general $d_n^f$ is topologically equivalent to $d$.
Given $Y\subset X$ (not necessarily $f$-invariant) we denote by $G_n^f(Y,\eps)$ the minimal number of balls of radius $\eps$ for the metric $d_n^f$ in a finite covering of $Y$ (the centers of the balls do not necessarily belong to $Y$). Where $Y=X$, we only write $G_n^f(\eps)$.

\begin{Def}
The \emph{polynomial entropy} $\hp(f,Y)$ of $f$ on $Y$ is defined by
\begin{align*}
\hp(f,Y) &=\sup_{\eps}\inf\left\{\sigma>0|\limsup\frac{1}{n^{\sigma}}G_n^f(Y,\eps)=0\right\}\\
& =\lim_{\eps\rit 0}\limsup_{n\rit\infty}\frac{\log G_n^f(Y,\eps)}{\log n}.
\end{align*}
When $Y=X$, we just write $\hp(f)$, and we call it the  polynomial entropy of $f$.
\end{Def}

Instead of balls of radius $\eps$, we can consider sets with diameter smaller than or equal to $\eps$ for the metric $d^f_n$. We denote by $D_n(Y,\eps)$ the smallest number of sets $X_i$ such that
\[
 Y \subset \bigcup_i X_i \quad\textrm{and}\quad {\rm diam}_{d_f^n}X_i\leq\eps.
\]
We also consider sets that are $\eps$\emph{-separated}  for the metrics $d_n^f$ (we will write $(n,\eps)$-separated). Recall that a set $E$ is said to be $\eps$-separated for a metric $d$ if  for all $(x,y)$ in $E^2$, $d(x,y)\ge \eps$. 
Denote by  $S_n^f(Y,\eps)$ the maximal cardinal of a $(n,\eps)$-separeted set contained  in $Y$. When $Y=X$, we only write $D_n^f(\eps)$ and $S_n^f(\eps)$.
Observe that: 
\[
D_n^f(Y,2\eps)\leq G_n^f(Y,\eps) \leq D_n^f(Y,\eps) \quad\textrm{and}\quad
S_n^f(Y,2\eps)\leq G_n^f(Y,\eps)\leq S_n^f(Y,\eps).
\]
Therefore
\[
\hp(f,Y) =\lim_{\eps\rit 0}\limsup_{n\rit\infty}\frac{\log D_n^f(Y,\eps)}{\log n}=\lim_{\eps\rit 0}\limsup_{n\rit\infty}\frac{\log S_n^f(Y,\eps)}{\log n}.
\]

\begin{rem}
If $\phi:=(\phi^t)_{t\in \R}$ is a continuous flow on $X$, for $t>0$ and $\eps>0$,  we can define in the same way the numbers $G_t^\phi(\eps), D_t^\phi(\eps)$ and $S_t^\phi(\eps)$.
The polynomial entropy  $\hp(\phi,Y)$ of $\phi$ on $Y\subset X$ is defined as 
\[
\hp(\phi,Y) =\lim_{\eps\rit 0}\limsup_{t\rit\infty}\frac{\log D_t^f(Y,\eps)}{\log t}=\lim_{\eps\rit 0}\limsup_{t\rit\infty}\frac{\log S_t^f(Y,\eps)}{\log t}.
\] 
One easily checks that if $\phi^1$ is the time-one map of $\phi$, $\hp(\phi)=\hp(\phi^1)$.
\end{rem}

The following properties of the polynomial entropy are proved in  \cite{Mar-09}. 
\begin{pte}\label{ptehphw}
\begin{enumerate} 
\item $\hp$ is a $C^0$ conjugacy invariant, and does not depend on the choice of topologically equivalent
metrics on $X$.
\item If $Y\subset X$ is \emph{invariant} under $f$ and endowed with the induced metric, then $\hp(f,Y)=\hp(f_{\vert Y})$.
\item If $A$ is a subset of $X$ invariant under $f$, $\hp(f_{\vert_A})\leq \hp(f)$.
\item For $m\in\N$, $\hp(f^m) =\hp(f)$ and if $f$ is invertible, $\hp(f^{-m})=\hp(f)$.
\item If $A=\cup_{i=1}^n A_i$ where $A_i$ is $f$-invariant, $\hp(f_{\vert_A}) = \max_i(\hp(f_{|_{A_i}})$.
\end{enumerate}
\end{pte}

\begin{prop}\label{pointerrant}
Let $(X,d)$ be a compact metric space and let $f: X\rit X$ be a homeomorphism on $X$. Assume that $f$ possesses a wandering point. Then $\hp(f)\geq 1$.
\end{prop}

\begin{proof}
Let $x$ be a wandering point for $f$. Then, there exists $\eps_0>0$ such that for all $k\in \N$, $f^k(B(x,\eps_0))\cap B(x,\eps_0)=\emptyset$.
Then, for all $\eps \leq \eps_0$, and all $k\in \N$, $d(x,f^k(x))\geq \eps$. 
Fix $n\in \N$. For $1\leq k \leq n$, $d_n^f(x, f^{-k}(x)) \geq d(f^{k}(x), x) \geq \eps$.
So for all $\eps\leq \eps_0$ and all $n\in \N$, the points $x, f\inv(x),\dots, f^{-n}(x)$ are $(n,\eps)$--separated,  and $S_n^f(\eps)\geq n$ which proves that $\hp(f)\geq 1$.
\end{proof}

\section{Homeomorphisms of the circle}

We denote by $\Hom^+(\T)$ the group of orientation-preserving homeomorphisms of the circle $\T$. If $f\in \Hom^+(\T)$, there exists a continuous increasing function $F: \R \rit \R$ such that $\pi\circ F=f\circ \pi$, where $\pi :\R \rit \T$ is the canonical projection. 
We say that $F$ is a lift of $f$ to $\R$. Such a lift is unique up to an integer.
The most simple example of such a homeomorphism is the one of a rotation. Obviously, the polynomial entropy of a rotation is zero.
%
Poincaré proved that if $F$ is the lift of $f\in \Hom^+(\T)$, there exists a unique number $\rho(F)\in \R$ such that for all $x\in \R$, $\lim_{n\rit\pm \infty} \frac{F^n(x)}{n}=\rho(F)$. This number is called the \emph{rotation number of $F$}.
For any $k\in \Z$, $\rho(F+k)=\rho(F)+k$, so the class in $\T$ of the rotation number $\rho(F)$ of a lift $F$ of $f\in \Hom^+(\T)$ is independent of the choice of the lift. This class is called the \emph{rotation number of $f$}.

For $a\in \R$, we set $T_a : \R \rit \R : x \ma x+a$. The translation $T_a$ is a lift of the \emph{rotation of angle $a$} $\tau_a : \T  \rit \T : \bar{x} \ma \bar{x}+ \bar{a}$, where $\bar{a}=\pi(a)$. 
The rotation number satisfies the following properties.
\begin{pte}
\begin{enumerate}
\item For any $a\in \R$, $\rho(T_a)= a$.
\item If $f\in \Hom^+(\T)$, for any lift $F$ of $f$ and any $n\in \N^*$, $\rho(F^n)=n\rho(F)$.
\end{enumerate}
\end{pte}
The knowledge of the rotation number of $f\in \Hom^+(\T)$ permit to describe the behavior of the dynamical system generated by $f$. Indeed this behavior depends on whether $\rho(f)$ is rational or irrational.
In Section 3.1 we remind the dymamics behavior when the rotation number is rational and we prove Theorem  \ref{cercle} in this case (Proposition \ref{rationel}). In Section 3.2  we remind the dymamics behavior when the rotation number is irrational and prove Theorem  \ref{cercle} in this second case (Proposition \ref{irrationel}).

\subsection{Homeomorphisms with rationnal rotation number}

An element $f$ of $\Hom^+(\T$) has rational rotation number if and only $f$ possesses periodic orbits, and it has zero rotation number if and only if $f$ possesses  fixed points.

We begin with the case when $\rho(f)=0$. We denote by $\Fix(f)$ the set of fixed points of $f$. If $f\neq \Id$, then $\T\setm \Fix(f)$ is a nonempty open domain of $\T$ whose connected component are invariant by $f$. 
Moreover if $I=\,]\bar{a}, \bar{b}[$ is such a connected component, the $\om$--limit set of $I$ is one of the extremities of $I$ and the $\al$--limit set is the other one.

Assume now that $\rho(f)=\pq$ where $p,q$ are two coprime integers. Then all the periodic orbits have a common smallest period equal to $q$. Moreover, $f$ is conjugate to a rotation if and only if all its orbit are periodic. Otherwise, the set $\Per(f)$ of the periodic points of $f$ is a proper compact subset of $\T$.
Finally, observe that if $\rho(f)=\pq$, then $\rho(f^q)=0$. Indeed $\Per(f)=\Fix(f^q)$. In particular, if $f$ is not conjugate to a rotation, $\T\setm \Per(f)$ is a nonempty open domain whose connected component are invariant by $f^q$.
The main result of this section is the following.

\begin{prop}\label{rationel}
Let $f\in \Hom^+(\T)$ such that $\rho(f)\in \Q$. Assume that $f$ is not conjugate to a rotation. Then 
$
 \hp(f)=1.
$
\end{prop}

The proof is based on the following lemma.

\begin{lem}\label{prop:indseg} 
Let $I=[a,b]$ be a compact interval in $\R$ and let $f : I\rit I$ be a continuous,  increasing function such that $f(a)=a$, $f(b)=b$ and $f(x)-x\neq 0$ for all $x\in\, ]a,b[$. Then
$
\hp(f)=1.
$
\end{lem}

\begin{proof}
Changing $f$ by $f\inv$ is necessary, we can assume without loss of generality that $f(x)>x$ for all $x\in\,]a,b[$.
We first observe that any $x\in\,]a,b[$ is wandering. Therefore by Proposition \ref{pointerrant}, $\hp(f)\geq 1$.
We will now prove that $\hp(f)\leq 1$. Given $\eps>0$ we will construct for $n$ large enough a majoration of the form $D_n^f(\eps)\leq cn+d$, where $c$ and $d$ are positive numbers independent of $n$.
Fix $0<\eps<\frac{b-a}{2}$. 
We introduce the intervals $I_0=[a,a+\eps]$, $I_1=[b-\eps,b]$ and $J=[a+\eps, b-\eps]$. 
\vskip 2mm 

\noindent $\bullet$ \textit{Cover of} $I_1$: Observe that for all $n\in \N$, $f^n(I_1)=I_1$, so $D_n^f(I_1,\eps)=1$.

\vskip 2mm  

\noindent $\bullet$ \textit{Cover of} $J$: Since $\lim_{n\rit \infty}f^n(a+\eps)=b$, there exists $n_0$ such that for $n\geq n_0$, $f^{n_0}(a+\eps)\in I_1$. 
So for $n\geq n_0$, $f^n(J)\subset I_1$ and in particular $f^n(J)$ has diameter less than $\eps$.
Now, $f, f^2,\dots, f^{n_0}$ are uniformly continuous, so there exists $\eta>0$ such for all $(x,y)\in [a,b]^2$ if $|x-y|\leq \al$, then, for all $1\leq k \leq n_0$,  $|f^k(x)-f^k(y)|\leq \eps$.
We divide $J$ in a finite number $q$ of subintervals $J_1,\dots, J_q$ with diameter $\al$. Observe that $q$ depends only on $\eps$. 
Now when $n\geq 1+n_0$, $f^n(J_k)\subset I_1$ so $f^n(J_k)$ has diameter less than $\eps$. Therefore, we get a covering of $J$ with $q$ sets of $d_n^f$-diameter less than $\eps$ for all $n\in \N$.

\vskip 2mm  

\noindent $\bullet$ \textit{Cover of} $I_0$. Fix $n\geq 1$. First, observe that the interval $[a+\eps,f(a+\eps)]$ is covered by some intervals of the previous family, say $J_1,\ldots,J_p$, with $1\leq p\leq q$. 
Thus the interval $[f\inv(a+\eps),a+\eps]$ is covered by $f\inv(J_j)$, $1\leq j\leq p$.
By construction, for any $N\in \N$, $f^N(f\inv(J_j))$ has diameter less than $\eps$, so $[f\inv(a+\eps),a+\eps]$ can be covered 
by  $p$ sets with $d_N^f$-diameter less than $\eps$ for any $N\in \N$. 

By the same argument, taking the pullbacks of order $k$, $1\leq k\leq n$ of $[a+\eps,f(a+\eps)]$, we obtain a cover of each interval 
$[f^{-(k+1)}(a+\eps),f^{-k}(a+\eps)]$ by $p$ balls. 
Thus we get a covering of $[f^{-n}(a+\eps),a+\eps]$ by $np$ sets of $d_N^f$-diameter less than $\eps$ for any  $N \in \N$. In particular, $D_n^f([f^{-n}(a+\eps),a+\eps],\eps) \leq np$.

It only remains to cover the interval $[a,f^{-n}(a+\eps)]$. For $x\in [a,f^{-n}(a+\eps)]$ and $1\leq k \leq n$, $f^k(x)\leq f^k(f^{-n}(a+\eps))\leq \eps$, that is, $f^k([a,f^{-n}(a+\eps)])\subset [0,\eps]$ and $D_n^f([a,f^{-n}(a+\eps)],\eps)=1$.

Gathering these three results, we get a covering of $I_0$ with $np+1$ sets of $d_n^f$-diameter less than $\eps$.

\vskip 2mm  

Hence, for $n\in \N^*$, we obtain a cover of $[a,b]$ by $np+q+2$ sets of diameter $\eps$ for $d_n^f$, with $p$ and $q$ independent of $n$. This proves that
$
D_n^f(\eps)\leq np+q+2,
$
and therefore that $\hp(f)\leq 1$. So $\hp(f)=1$.
\end{proof}

%

\begin{proof}[Proof of proposition \ref{rationel}]
Let $(p,q)\in \Z\times \N^*$ with $p\wedge q=1$ be such that $\rho(f)=\pq$. By property \ref{ptehphw} (4), it suffices to compute $\hp(f^q)$.

Let $F$ be the lift of $f$ to $\R$ such that for all $x\in \pi\inv (\Per(f)), F^q(x)=x+p$ and let $\wti G=F^q-p$. The function $\wti G$ is a lift of $f^q$ and $\pi\inv(\Per(f))=\Fix \wti G=\pi\inv(\Fix (f^q))$.

Fix $\bz_0 \in \Per(f)$, fix $z_0\in \pi\inv(\{\bz_0\})$ and consider the restriction $G$ of $\wti G$ to the interval $[z_0, z_0+1]$.
Set $P:=\Fix G= \Fix \wti G\cap [z_0, z_0+1]$, it is a compact subset of $[z_0, z_0+1]$. We set $P:=\{z_i\,|\,i\in \II\}$ where $\II\subset \R$.

Let $\Ii$ be a connected component of $\T\setm \Per(f)$ with closure $\ov{\Ii}$. Set $I=\pi\inv(\Ii)\cap [z_0, z_0+1]$. Then $\ov{I}:=\pi\inv(\ov{\Ii})\cap [z_0, z_0+1]$ and $\pi: \ov{I}\rit \ov{\Ii}$ is an isometry that conjugates $G$ and $f^q$.
Now $G$ is  increasing on $I$,  fixes its extremities and $G-\Id$ never vanishes on $\overset{\circ}{\Ii_k}$. In particular, for all $x\in \overset{\circ}{I}$, $x$ is wandering for $G$, and therefore $\pi(x)$ is wandering for $f^q$. Then by Proposition \ref{pointerrant}, $\hp(f^q)\geq 1$.

We will now prove that $\hp(f^q)\leq 1$.  
For this, given $\eps>0$ we will construct for $n$ large enough a majoration of $D_n^{f^q}(\eps)$ of the form $D_n^{f^q}(\eps)\leq cn+d$, where $c$ and $d$ are positive numbers independent of $n$.
Let $\eps>0$. 
Since $P$ is compact, there exists  a finite subset $\{z_{i_0},\dots, z_{i_\ka}\}$ of $P$ such that the closed balls $B_k:=B(z_{i_k},\frac{\eps}{2})$ cover $P$.
Let $\Bb_k=\pi(B_k)$. The sets $\Bb_k$ cover $\Per(f)$ and have diameter less than $\eps$. 

Fix $k\in\{1,\dots, \ka\}$ and set $z_m^k=\min B_k\cap P$ and $z_M^k=\max B_k\cap P$.
We set $J_k:=[z_m^k, z_M^k]$ and $\Jj_k=\pi(J_k)$. 
Observe that $J_k$ is invariant by $G$, that $\Jj_k$ is invariant by $f^q$ and that $\pi : J_k\rit \Jj_k$ is an isometry that conjugates $G$ and $f^q$.
We choose the indices such that for $0\leq k \leq \ka-1$, $z_M^k <z_m^{k+1}$. 
In particular, $z_m^0=z_0$ and $z_M^\ka=z_0+1$.

For $1 \leq k \leq \ka$, we set $I_k:=]z_M^{k-1},z_m^k[$ and  $\Ii_k=\pi(I_k)$. We denote by $\ov{I_k}$ and $\ov{\Ii_k}$ their respective closure.
\vspace{0.2cm}

\noindent $\bullet$ \textit{Cover of $\Jj_k$:} Since the projection $\pi : J_k\rit \Jj_k$ is an isometry that conjugates $G$ and $f^q$ it suffices to find a cover of $J_k$ by sets with $d_n^G$--diameter less than $\eps$.
By construction, $\diam J_k\leq \eps$. 
Now $J_k$ is invariant by $G$, so for any $n\in \N$, $G^n(J_k)$ has diameter less than $\eps$.
Therefore, $D_n^{f^q}(\Jj_k,\eps)=D_n^G(J_k,\eps)=1$.
\vspace{0.2cm}

\noindent $\bullet$ \textit{Cover of $\ov{\Ii_k}$:} As before,  it suffices to find a cover of $I_k$ by sets with $d_n^G$--diameter less than $\eps$.
Observe each $\Ii_k$ is a connected component of $\T\setm\Per(f)$, so we have already seen  that the restriction of $G$ to $\ov{\Ii_k}$ satisfies the hypotheses of Proposition \ref{prop:indseg}.
 Therefore, using  the proof of proposition \ref{prop:indseg}, one sees that there exists two postive numbers $c_k, d_k$ independent of $n$ such that $D_n^{f^q}(\Ii_k,\eps)\leq D_n^{f^q}(\ov{\Ii_k},\eps)=D_n^G(\ov{I_k},\eps)\leq c_kn +d_k$
\vspace{0.2cm}

Set $c:=\sum_{j=1}^k c_j$ and $d=\sum_{j=1}^k d_j$. Then, $D_n^{f^q}(\eps)\leq cn +d +\ka$ which yields $\hp(f^q)\leq 1$. Therefore, $\hp(f^q) =1$.
\end{proof}


\subsection{Homeomorphisms with irrationnal rotation number}

Let $f\in \Hom^+(\T)$ with irrational rotation number. One proves that the $\al$ and $\om$--limit sets of $\bar{x}\in \T$ do not depend on $\bar{x}$. We denote by $\Om$ this set. It is a closed set invariant by $f$. Moreover, one proves that $\Om$ is the set of nonwandering points of $f$.
Finally, $f$ is conjugate to an irrational rotation if and only if  $\Om=\T$. Otherwise, $\Om$ is a Cantor set of $\T$ and any connected component of $\T\setm \Om$ is a wandering interval of $\T$.

\begin{prop}\label{irrationel}
Let $f\in \Hom^+(\T)$ such that $\rho(f)\in \R\setm \Q$. Assume that $f$ is not conjugate to a rotation. Then 
$
 \hp(f)=1.
$
\end{prop}

\begin{proof}
We first observe that if $f$ is not conjugate to a rotation then it admits wandering points. Then $\hp(f)\geq 1$. 
Let us prove that $\hp(f)\leq 1$. Again, for $\eps>0$ fixed and $n$ large enough, we will prove that $D_n^f(\eps) \leq cn+d$ with $c, d$ independent of $n$.
We denote by $\pi$ the canonical projection $\pi :[0,1]\rit \T$, it is continuous and injective over $[0,1[$. 
We denote by $\II:=\T\setm \Om$ the complementary of  the  Cantor set $\Om$ of nonwandering points of $f$.
We can assume without loss of generality that $\pi(0)\in \II$.
We set $\wti{\Om}:=\pi\inv(\Om)$ and $\wti\II:=\pi\inv(\II)$.

Fix $\eps>0$ such that $\eps<\demi$. We begin by constructing a cover of $\Om$ by pairwise disjoint connected sets with diameter less than $\eps$.
To do this, we construct a finite increasing sequence $a_1<\dots<a_{p-1}<a_p$ of points of $\wti{\II}$ in the following way.
We set $a_1=0$.
Since $\wti{\Om}$ is totally disconnected,  the intersection $]\ee, \eps[\,\cap \wti{\II}$ is nonempty, and  $a_2$ is chosen to be an arbirary point in $]\ee, \eps[\,\cap \wti{\II}$.
Similarly, we choose $a_3$ such that $a_3\in\,]a_2 +\ee, a_2+\eps[\,\cap \wti \II$.
By a recursive process we construct $a_{k+1}$ such that $a_{k+1}\in \,]a_k +\ee, a_k +\eps[\,\cap \wti{\II}$. 

By construction, the intervals $[a_k,a_{k+1}]$ have length more than $\ee$. So there exists a minimal $p\in \N$ such that $[0,1]\subset \cup_{k=1}^p[a_k,a_{k+1}]$.

For $1\leq k\leq p$, the sets $\wti\Om\cap [a_k,a_{k+1}]$ are compact. This allows us to set $\om_m^k:=\min (\wti\Om\cap [a_k,a_k+1])$ and $\om_M^k:=\max (\wti\Om\cap [a_k,a_k+1])$. 
Setting, $B_k:=[\om_m^k,\om_M^k]$, the family  $(B_k)_{1\leq k\leq p}$ is a finite cover of $\wti\Om$ by pairwise disjoint connected sets with diameter less than $\eps$. 
Then, setting $\Bb_k:= \pi([\om_m^k,\om_M^k])$, we get a cover of $\Om$ with the desired form.

We set $I_0:=\pi([0,\om_1^m[\,\cup\,]\om_p^M,1])$ and $I_k:=\pi(]\om_k^M, \om_{k+1}^m[)$, for $1\leq k \leq p$. So for $0\leq k \leq p$, the intervals $I_k$ are connected components of $\II$.
Set $\BB:=\cup_{k=1}^{p}\Bb_k$ and $\II:=\cup_{k=0}^{p} I_k$.
\vspace{0.2cm}

\noindent $\bullet$ \textit{Cover of $\II$:} Since each $I_k$ is wandering, there exists $n_0$, such that for all $N\geq n_0$ and all $0\leq k \leq p$, $f^N(I_k)\subset \BB$. Since $I_k$ is connected, for all $N\geq n_0$, $f^N(I_k)$ is contained in one of the intervals $\Bb_1,\dots, \Bb_p$, so for $N\geq n_0$, $f^N(I_k)$ has diameter less than $\eps$.

Now, $f, f^2,\dots, f^{n_0}$ are uniformly continuous, so there exists $\al>0$ such that for all $(x,y)\in \T^2$ if $d(x,y)\leq \al$, then  $d(f^N(x),f^N(y))\leq \eps$ for all $1\leq N \leq n_0$.
For $1\leq k \leq p$, we divide $I_k$ in a finite number $m_k$ of subintervals $I_k^1,\dots, I_k^{m_k}$ with length $\al$. Observe that each $m_k$ depends only on $\eps$. 
Setting  $m:=\max\{m_k\,|\, 1\leq k \leq q\}$,  we get a cover of $\II$ by at most $m(p+1)$ sets with $d_N$-diameter less than $\eps$ for any $N\in \N$. 

\vspace{0.2cm}

\noindent $\bullet$ \textit{Cover of $\BB$:} The set $\BB$ is the disjoint union of $\Om$ and of a countable union of wandering intervals $(J_i)_{i\in \N}$ with diameter less than $\eps$.
Fix $n\geq n_0$.

There exists  at most $p+1$ wandering intervals $J$ contained in $\Bb$ such that $f(J)\subset \II$. 
We denote them by $J_{i_1}^1,\dots, J_{i_{q_1}}^1$, with $q_1\in \{0,\dots,p\}$.
For each $i_j$, there exists $k\in\{0,\dots,p\}$ such that $f(J_{i_j}^1)=I_k$.
So $J_{i_j}=\cup_{\ell=1}^{m_k}f\inv(I_k^{\ell})$ and for $0\leq N \leq n_0+1$, $f^N(f\inv(I_k^{\ell}))$ has diameter less than $\eps$.
On the other hand, for all $N \geq n_0+1$, $f^N(f\inv(I_k^{\ell}))\subset f^N(J_{i_j}^1)\subset \Bb$ and by the same connectedness argument as before $f^N(f\inv(I_k^{\ell}))$ has diameter less than $\eps$.
This way, we get a cover of $J_{i_j}$ by at most $m_k\leq m$ sets with $d_N$-diameter less than $\eps$ for any $N\in \N$. So the union $\Jj^1$ of these intervals is covered by at most $m(p+1)$ sets  with $d_N$-diameter less than $\eps$ for any $N\in \N$

Similarly there exists at most $p+1$ wandering intervals $J$ contained in $\BB$ such that $f^2(J)\subset \II$ and $f(J)\subset \BB$. 
We denote by $\Jj^2$ their union.
The domains $\Jj_1$ and $\Jj_2$ are disjoint.
As before, each wandering interval $J$ contained in $\Jj^2$ is covered by the intervals $f^{-2}(I_k^{\ell}), 1\leq \ell \leq m_k$, for some $k$ and one checks that for all $N\in \N$, $f^N(f^{-2}(I_k^{\ell}))$ has diameter less than $\eps$.
Therefore we get again a cover of $\Jj^2$ by at most $m(p+1)$ sets with $d_N$-diameter less than $\eps$ for any $N\in \N$.

In general, for $n\in \N$, there exists at most $p+1$ wandering intervals $J$ contained in $\Bb$ such that $f^n(J)\in \Ii$ and $f^{\ell}(J)\subset \Bb$ when $0\leq \ell\leq n-1$. 
We denote by $\Jj^n$ their union.
Again, considering the intervals $f^{-n}(I_k^\ell)$ for $0\leq k \leq p$ and $1\leq \ell \leq m_k$,  we get a cover of $\Jj^n$ by at most $m(p+1)$ sets with $d_N$-diameter less than $\eps$ for any $N\in \N$.

Fix $n\in \N$. Set $\JJ:=\cup_{k=1}^n \Jj^k$. By the previous study, we get a cover of $\JJ$ by at most $nm(p+1)$ sets with $d^N$--diameter less than $\eps$ for all $N\in \N$. 


It remains to cover $\Bb\setm \JJ$. Let $C$ be a connected component of $\BB\setm \JJ$. 
Then, by definition of $\JJ$, for all $N\leq n$, $f^N(C)\subset \Bb$ and is therefore contained in one of the intervals $\Bb_k$. 
Therefore, $C$ has $d_n$-diameter less than $\eps$.
Now, $\JJ$ has at most $n(p+1)$ connected components, so $\BB\setm \JJ$ has at most $(n(p+1)+1)p$ connected components.
Thus we get a cover of $\BB\setm \JJ$ by at most $(n(p+1)+1)p$ sets with $d_n$-diameter less than $\eps$.

The previous cover yields $D_n^f(\eps)\leq  n((p+1)(m +p))+ p +m(p+1)$, with $m,p,q$ depending only on $\eps$. This yields $\hp(f)\leq 1$.
\end{proof}

\section{$C^1$ nonvanishing vector fields on $\T^2$}

Let $X :\T^2\rit \R^2$ be a $C^1$ vector field on $\T^2$ that does not vanish. 
To compute the polynomial entropy of its flow, we discriminate between the cases where there are no periodic orbits (Proposition \ref{without} in Section 4.2) and when there are periodic orbits (Proposition \ref{periodicorbit} in Section 4.3).
We begin with a lemma of comparison of solutions of some differential equations  that will be useful for the estimation of the polynomial entropy. 

\subsection{A $L^1$ comparison lemma}

\begin{lem}[$L^1$ comparison lemma]\label{lemmecomparaisonL1}
Let $I=[\al,\be]\subset\R$ and consider two continuous functions $f,g: I\rit \R$ such that there exist
$m<M$ in $\R^{+*}$ with
\begin{itemize}
\item $m\leq f(x) \leq M$ and $m\leq g(x) \leq M$ for all $x\in \R$,
\item $\max(\norm{f-g}_{L_1(I)},\norm{f-g}_{\infty})<m$
\end{itemize}
where $\norm{\ }_{L^1(I)}$ and $\norm{\ }_\infty$ stand for the usual norms on $I$.
Consider the following differential equations:
\[
 (E_1): \dot{x}= f(x) \quad {\rm{and}} \quad (E_2) : \dot{x}=g(x).
\]
Fix $x_0\in I$ and let $\ga$ and $\eta$ be the solutions of $(E_1)$ and $(E_2)$ with inital conditions $\ga(0)=\eta(0)=x_0$. 
Then $\ga$ and $\eta$ are defined on subintervals $J_\ga$ and $J_\eta$ of $I$ respectively and for all 
$t\in J_\ga\cap J_\eta$:
\[
|\ga(t)-\eta(t)|\leq \frac{M}{m^2}\norm{f-g}_{L_1(I)}.
\]
\end{lem}

\begin{proof}
Let $F$ and $G$ be the primitives of $\frac{1}{f}$ and $\frac{1}{g}$ that vanish at $x_0$. 
Since $\frac{1}{f} \geq \frac{1}{M}$ and  $\frac{1}{g}\geq \frac{1}{M}$ with $\frac{1}{M}>0$, 
$F$ and $G$ are $C^1$ diffeomorphisms
on their images $J_\ga$ and $J_\eta$, which moreover satisfy  $F\inv=\ga$ and $G\inv=\eta$.
Indeed, for $t\in \R$,
$
(F\inv)'(t)=\frac{1}{F'(F\inv(t))}=f(F\inv(t)), 
$
and similaraly
$
(G\inv)'(t)=\frac{1}{G'(G\inv(t))}=g(G\inv(t)), 
$
and
$F\inv(0)=G\inv(0)=x_0$ by definition. 
Therefore, for $t\in J_\ga\cap J_\eta$:
\begin{equation}
t=\int_{x_0}^{\ga(t)} \frac{du}{f(u)}=\int_{x_0}^{\eta(t)} \frac{du}{g(u)}.
\end{equation}
Set $\De:= g-f$.
Then 
\begin{equation}\label{egalitet1}
\int_{x_0}^{\ga(t)} \frac{du}{g(u)-\De(u)}=\int_{x_0}^{\ga(t)} \frac{du}{f(u)}=\int_{x_0}^{\eta(t)} \frac{du}{g(u)}=\int_{x_0}^{\eta(t)} \frac{du}{f(u)+\De(u)}.
\end{equation}
Fix $t\in J_\ga\cap J_\eta$ and assume first that $\eta(t)\geq \ga(t)$. By (\ref{egalitet1}), one has
\begin{equation}\label{egalite2}
\int_{x_0}^{\ga(t)}\frac{du}{f(u)}=\int_{x_0}^{\eta(t)} \frac{du}{f(u)\left(1+\frac{\De(u)}{f(u)}\right)}.
\end{equation}
Since $\norm{f-g}_{\infty}<m$, then $\left|\frac{\De(u)}{f(u)}\right|< 1$, and in particular $\frac{\De(u)}{f(u)}>-1$, so  
\[
\frac{1}{1+\frac{\De(u)}{f(u)}}\geq 1-\frac{\De(u)}{f(u)}.
\]
Hence
\[
\int_{x_0}^{\ga(t)}\frac{du}{f(u)}=\int_{x_0}^{\eta(t)}\frac{du}{f(u)\left(1+\frac{\De(u)}{f(u)}\right)}\geq \int_{x_0}^{\eta(t)}\frac{du}{f(u)}-\int_{x_0}^{\eta(t)}\frac{\De(u)}{f(u)^2}du.
\]
This yields
\begin{multline}
0 \leq \int_{\ga(t)}^{\eta(t)}\frac{du}{f(u)}\leq \int_{x_0}^{\eta(t)}\frac{\De(u)}{f(u)^2}\leq \frac{1}{m^2}\int_{x_0}^{\eta(t)}\De(u)du\\
=\frac{1}{m^2}\left|\int_{x_0}^{\eta(t)}\De(u)du\right|\leq \frac{1}{m^2}||\De||_{L_1(I)}.
\end{multline}
On the other hand,
$
\int_{\ga(t)}^{\eta(t)}\frac{du}{f(u)} \geq (\eta(t)-\ga(t))\frac{1}{M},
$
so 
\[
\eta(t)-\ga(t)\leq \frac{M}{m^2} ||\De||_{L_1(I)}.
\]

Assume now that $\eta(t) \leq \ga(t)$. By (\ref{egalitet1}), one has
\begin{equation}\label{egalite2}
\int_{x_0}^{\eta(t)} \frac{du}{g(u)}=\int_{x_0}^{\ga(t)} \frac{du}{g(u)\left(1-\frac{\De(u)}{g(u)}\right)}.
\end{equation}
As before, $\left|\frac{\De(u)}{f(u)}\right|< 1$, and in particular $\frac{\De(u)}{f(u)}<1$, so   
\[
\frac{1}{1-\frac{\De(u)}{g(u)}}\geq 1+\frac{\De(u)}{g(u)}.
\]
Hence,
\[
\int_{x_0}^{\eta(t)}\frac{du}{g(u)}=
\int_{x_0}^{\ga(t)}\frac{du}{g(u)\left(1-\frac{\De(u)}{g(u)}\right)}du
\geq \int_{x_0}^{\ga(t)}\frac{du}{g(u)}+\int_{x_0}^{\ga(t)}\frac{\De(u)}{g(u)^2}du.
\]
As before, this yields
\begin{multline}
0 \leq \int_{\eta(t)}^{\ga(t)}\frac{du}{g(u)}\leq \int_{x_0}^{\ga(t)}\frac{-\De(u)}{f(u)^2}\leq \frac{1}{m^2}\int_{x_0}^{\ga(t)}-\De(u)du\\
=\frac{1}{m^2}\left|\int_{x_0}^{\ga(t)}-\De(u)du\right|\leq \frac{1}{m^2}||\De||_{L_1(I)}.
\end{multline}
Now, as before 
$
\int_{\eta(t)}^{\ga(t)}\frac{du}{g(u)} \geq (\ga(t)-\eta(t))\frac{1}{M},
$
so 
\[
\ga(t)-\eta(t)\leq \frac{M}{m^2}\norm{\De}_{L_1(I)},
\]
which concludes the proof.
\end{proof}

Consider now a bounded Lipschitz vector field $X=(X_1, X_2)$ on $\R^2$ with  flow $\phi_X$.
Set $\mu_M:=\max_{\R^2}  X_1$ and $\mu_m:=\min_{\R^2} X_1$.
We moreover assume that $\mu_m>0$.

\begin{lem}\label{orbitegraphe}
The orbits of $\phi_X$ are graphs over the $x$-axis.
\end{lem}

\begin{proof}
Since $X$ is bounded, $X$ is complete. Let $\ga :\R\rit \R^2$ be a maximal solution of $X$. We set $\ga : t \mapsto (\ga_x(t), \ga_y(t))$. 
It suffices to show that  $t\ma \ga_x(t)$ is a diffeomorphism.
Since $\ga_x'(t)= X_1(\ga(t))\geq \mu_m>0$, $\ga_x$ is injective and is a diffeomorphism on its image.
Now for $t>0$, $\ga_x(t)\geq \mu_m t$ so $\lim_{t\rit +\infty} \ga_x(t)=+\infty$. Similarly, $\lim_{t\rit -\infty} \ga_x(t)=-\infty$, so $\ga_x(\R)= \R$. 
\end{proof}

In the following we consider two functions $\vp, \psi : \R \rit \R$ such that their graphs $G_\vp:=\{(x,\vp(x))\,|\, x\in \R\}$ and $G_\psi:=\{(x,\psi(x))\,|\, x\in \R\}$ are orbits of $X$.
We moreover assume that for all $x\in \R$ $\psi(x) > \vp(x)$.

\begin{Def} For $x_0\in \R$, we set $V_{x_0}(\vp,\psi):= \{(x_0,y)\,|\, y\in [\vp(x_0), \psi(x_0)]\}$.
The \emph{deviation of the vertical $V_{x_0}(\vp,\psi)$ at time $t$}  is defined as
\[
\Dev(V_{x_0}(\vp, \psi),t):= \max_{y\in[\vp(x_0), \psi(x_0)]} |\pi_x(\phi_X^t(x_0, \vp(x_0)))-\pi_x(\phi_X^t(x_0,y))|
\]
where $\pi_x:\R^2\rit \R$ is the canonical projection on the $x$-axis.
\end{Def}

\begin{prop}\label{L1Deviation} 
There exists $C>0$, depending only on $X$, such that for $x_0\in\R$ and
$T>0$,  
\[
\Dev(V_{x_0}(\vp, \psi),T)\leq C\norm{\vp-\psi}_{L^1([0,\mu_M T])}.
\]
\end{prop}

\begin{proof} 
This is a consequence of the $L^1$ comparison lemma. We denote by $\OO(\vp, \psi)$ the set of functions $f:\R\rit \R$ such that the graph $G_f$ of $f$ is an orbit of $\phi_X$ and such that $\vp\leq f \leq \psi$.
For $f\in \OO(\vp, \psi)$ we denote by $X_f$ the vector field on $\R$ defined by 
$
X_f(x)= X_1(x,f(x))
$
and by $\ga_f$ the solution of  $\dot{x}= X_f(x)$  with initial condition $x_0$.
By construction $\ga_f(t)=\pi_x(\phi_X^t((x,f(x)))$ and
$
\ga_f(t)\in[x_0,x_0+T\mu_M] 
$
for all $t\in[0,T]$.
Let $\ka$ be the  Lipschitz constant of $X_1$. Then, for $x\in \R$,
\[
\abs{X_\vp(x)-X_f(x)}\leq \ka\abs{\vp(x)-f(x)}.
\]
By the $L^1$ comparison lemma, for $t\in[0,T]$
\[
|\ga_\vp(t)-\ga_f(t)|
\leq \frac{\mu_M}{\mu_m^2}\norm{X_\vp-X_f}_{L^1([0,\mu^* T])}
\leq \ka \frac{\mu_M}{\mu_m^2}\norm{\psi-\vp}_{L^1([0,\mu^* T])}.
\]
This concludes the proof with $C=\ka \frac{\mu_M}{\mu_m^2}$.
\end{proof}

\begin{cor}\label{corollaire} With the same assumptions and notation as in Proposition~\ref{L1Deviation},
there exist positive constants $c_0$, $c_1$ that depend only on $X$ such that for all $\eps>0$ and all $T>0$
if 
\[
\norm{\psi-\vp}_{C^0([0,\mu_MT])}<\frac{\eps}{3},
\]
\[
\norm{\psi-\vp}_{L^1([0,\mu_MT])}\leq c_1\eps,
\]
and 
\[
\abs{x-x'}\leq c_0\eps,
\] 
then the domain 
$D(x,x',\vp, \psi)$ bounded by the verticals $V_x(\vp, \psi), V_{x'}(\vp, \psi)$ and by the graphs $G_\vp$ and $G_\psi$ has 
$d_T^{\phi_X}$-diameter less than $\eps$.
\end{cor}

\begin{proof}
Set $\al:= \frac{\max X_1}{\min X_1}$ and $\be= \max( 1, \max X_1)$. 
Let $\eps >0$ and $t\in [0, \mu_MT]$.
Assume that $\norm {\psi-\vp}_{L^1([0,\mu_MT])}\leq \frac{\eps}{9\be C}$.
Then by the previous proposition, for all $x_0\in \R$,  and for $0 \leq t \leq T$, 
\begin{equation}\label{devdomaine}
\Dev(V_{x_0}(\vp, \psi),t)\leq C\norm{\psi-\vp}_{L^1([0,\mu_Mt])}\leq C\norm{\psi-\vp}_{L^1([0,\mu_MT])}\leq \frac{\eps}{9\be}.
\end{equation}
Let $x<x'$ in $\R$ with $x'-x \leq \frac{\eps}{9\be\al}$.
Let $\tau$ be such that $\pi_x(\phi_X^\tau((x, \vp(x)))=x'$. Then 
\[
\tau \leq \frac{x'-x}{\min X_1}\leq \frac{\eps}{9\al \be}\frac{1}{\min X_1}.
\]
Now, 
\begin{equation}\label{distancexx'}
0<\pi_x(\phi_X^t((x', \vp(x')))-\pi_x(\phi_X^t((x, \vp(x)))\leq \tau \max X_1\leq \frac{\eps}{9\al \be}\frac{\max X_1}{\min X_1}=\frac{\eps}{9\be}.
\end{equation}
Gathering (\ref{devdomaine}) and (\ref{distancexx'}), one sees that 
\[
\diam\pi_x(\phi_X^t(D(x_0, x_0', \vp, \psi)))\leq \frac{\eps}{3\be}.
\]
If $\norm{\psi-\vp}_{C^0}\leq \frac{\eps}{3}$, and since at any point $x\in \R$, the ``slopes'' of the graphs $G_\vp$ and $G_\psi$ take their values in $[-\be,\be]$, by triangular inequality, one checks that  
\[
\diam\pi_y(\phi_X^k(D(x_0, x_0', \vp, \psi)) \leq \frac{\eps}{3}+2\be \frac{\eps}{3\be}\leq \eps.
\]
This concludes the proof with $c_0=\frac{1}{9\be \al}$, $c_1=\frac{\eps}{9\be C}$.
\end{proof}

\subsection{Vector fields without periodic orbits}
This section is devoted to the proof of the following result.

\begin{propa}\label{without}
Let $X$ be a $C^1$ nonvanishing vector field on $\T^2$ without periodic orbits. Let $\phi_X$ be the flow of $X$. 
Then $\hp(\phi_X)\in [0,1]$.
\end{propa}

To begin with, we summarize some well-known results that can be found in \cite{God83}. Let $X$ be $C^1$ nonvanishing  vector field on the $\T^2$ and let $\phi_X=(\phi_X^t)_{t\in \R}$ be its flow. If there are no periodic orbits, the flow $\phi_X$ possesses a global closed transverse section $\Ga$. Such a section possesses a Poincaré return map $h$ and the flow $\phi_X$ is conjugate to a reparametrization of a suspension of $h$. 
Recall that a suspension of $h$ is a $C^1$ vector field $X_h$ on $\T^2$ such that its flow $\phi_h$ satisfies $\phi_h^1(z)=h(z)$ for all $z\in \Ga$. 
The Poincaré map $h$ is a diffeomorphism of the circle. Up to a change of sign on $X_h$, we can assume that $h$ is orientation preserving. The following lemma gives an explicit construction of a suspension.

\begin{lem}\label{suspension}
Let $f: \R\rit \R$ be a $C^1$ increasing diffeomorphism. There exists a vector field $\ha X_f$ on $\R^2$ satisfying $\ha X_f(x,y)=(1, X_2(x,y))$ such that $X_2(x+1,y)=X_2(x,y)$ for all $(x,y)$ and which defines a complete flow $(\ha\phi_f^t)_{t\in \R}$ on $\R^2$ such that :
\[
\ha\phi_f^1(n,y)=(n+1, f(y)), \quad \forall (n,y)\in \Z\times \R.
\]
\end{lem}

\begin{proof}
Let $\eta :]-\frac{1}{4},\frac{5}{4}[\rit [0,1]$ be a $C^\infty$ function such that $\eta\equiv 0$ on $]-\frac{1}{4},\frac{1}{4}[$ and $\eta\equiv 1$ on $]\frac{3}{4},\frac{5}{4}[$. 
Let $F:]-\frac{1}{4},\frac{5}{4}[\times \R\rit \R :(t,x)\ma (1-\eta(t))x+\eta(t)f(x)$. We set $f_t:=F(t, \cdot):\R\rit \R$. 
Then for all $t\in ]-\frac{1}{4},\frac{5}{4}[$, $f_t$ is a diffeomorphism and $F$ is an isotopy between $\Id_\R$ and $f$.

Consider the nonautonomous vector field on $\R$ defined by 
\[
Y(t_0,x_0)=\frac{d}{dt}f_{t_0+t}\circ f_{t_0}\inv(x_0)|_{t=0}.
\]
Then for $(t_0,x_0)\in\,]-\frac{1}{4}, \frac{5}{4}[\times \R$, the solution $\ga$ of the Cauchy problem $\dot{x}(t)=Y(t,x(t)), x(t_0)=x_0$ is given by 
\[
\ga(t)= f_t\circ f_{t_0}\inv(x_0)
\]
for $t$ close enough to $t_0$.
Indeed, 
\begin{align}
Y(t,\ga(t))& =\frac{d}{ds}f_{t+s}\circ f_t\inv(\ga(t))|_{s=0}\\
& =\frac{d}{ds}f_{t+s}\circ f_t\inv(f_t\circ f_{t_0}\inv(x_0))|_{s=0}\\
& = \frac{d}{ds}f_{t+s}\circ f_{t_0}\inv(x_0)|_{s=0}  =\dot{\ga}(t).
\end{align}
Consider now the vector field $X$ on $]-\frac{1}{4}, \frac{5}{4}[\times\R$ defined by $X(x,y)=(1, Y(x,y))$. By construction, the flow $\phi_X$ associated with $X$ is defined by 
\[
\phi_X^t(x,y)=(x+t, f_{x+t}\circ f_{x}\inv(y))
\]
In particular, $\phi_X^1(0,y)=(1, f(y))$.
Since for all $x\in\,]-\frac{1}{4}, \frac{1}{4}[$ and all $y\in \R$, $X(x,y)=X(x+1, y)=(1,0)$, one can define a continuation $X_f$ of $X$ in $\R^2$ by setting $X_f(x,y)= X(x-[x],y)$.
The vector field $X_f$ satisfies the required properties.
\end{proof}

In other words, one can always construct a suspension of $h$ as $C^1$ vector field $X_h$ on $\T^2$ which admits a lift $\ha X_h=(1,X_2)$ on $\R^2$ such that each vertical $V_m:=\{m\}\times\R$ is a global section for the flow $\ha\phi_h$ of $\ha X_h$, and such that the Poincaré map $V_m\to V_{m+1}$ is a lift of $h$ to $\R$. 
A reparametrized suspension of $h$ is a vector field of the form $\xi X_h$, where $\xi\in C^1(\T^2,\R^{*+})$ and $X_h$ is a suspension of $h$. 
By property \ref{ptehphw} 1), it sufficies to compute the polynomial entropy of the flow $\phi_h$ associated with $X_h$. To do this we will use Corollary \ref{corollaire}.

\begin{rem} 1) Such a lift $\ha X_h$ satisfies the assumptions of  the $L^1$ comparison lemma. 

\noindent 2) If $G_{\vp_0}:=\{(x,\vp_0(x)\,|\, x\in \R\}$  is an orbit of $\ha X_h$, then  $G_{\vp_0+1}:=\{(x,\vp_0(x)+1\,|\, x\in \R\}$ is also an orbit of $\ha X_h$, and the compact connected domain delimited by $G_{\vp_0}$, $G_{\vp_0+1}$, $V_0$ and $V_1$ is a fundamental domain for $\T^2$.
\end{rem}

Given two orbits $G_\vp$ and $G_\psi$ with $\vp<\psi$,  we denote by $S[\vp,\psi]$ the  strip of $\R^2$ bounded by $G_\vp$ and $G_\psi$, and by $S[\vp,\psi,T]$ the subdomain of $S[\vp,\psi]$ limited by the verticals $V_0$ and $V_T$. 
We denote by $A[\vp,\psi,T]$ the Lebesgue measure of $S[\vp,\psi,T]$.

\begin{lem}\label{minorationaire}Consider two functions $\vp<\psi$ such that their graphs $G_\psi$ and $G_\vp$ are orbits of $\ha X_h$. Assume that $\norm{\psi-\vp}_{C^0(\R)} <+\infty$. There exists $\nu$ depending only on $\ha X_h$, such that for any $T\geq \nu\norm{\psi-\vp}_{C^0(\R)}$,
\[
A[\vp,\psi,T]\geq \demi\norm{\psi-\vp}^2_{C^0([0,T])}.
\]
\end{lem}

\begin{proof} Let $p=\max_{\R^2}\left|\frac{X_2(x,y)}{X_1(x,y)}\right|$.
We first observe that since $G_\vp$ and $G_\psi$ are orbits of $\ha X$, for all $x\in \R$, $|\vp'(x)| \leq p$.
Let $x_0\in \R$.
By the Mean Value Theorem, for all $x\geq x_0$, 
\[
\psi(x)-\vp(x)\geq \psi(x_0)-\vp(x_0)-2p(x-x_0)
\]
 and for all $x\leq x_0$,
\[
\psi(x)-\vp(x)\geq \psi(x_0)-\vp(x_0)-2p(x_0-x).
\]
Then, if $|x-x_0| \leq \frac{1}{4p}(\psi(x_0)-\vp(x_0))$, 
\begin{equation}\label{TAF}
\psi(x)-\vp(x)\geq \demi (\psi(x_0)-\vp(x_0)).
\end{equation}

Let $T\geq \frac{1}{4p}\norm{\psi-\vp}_{C^0(\R)}$. Let $x_0\in [0,T]$ be such that $\norm{\psi-\vp}_{C^0([0,T])}= \psi(x_0)-\vp(x_0)$.
Let  $I\subset [0,T]$ be an interval with length $\frac{1}{4p}\norm{\psi-\vp}_{C^0([0,T])}$ containing $x_0$.
Then applying (\ref{TAF})
\[
A[\vp, \psi, T] \geq \int_I \psi(x)-\vp(x)dx \geq \demi \norm{\psi-\vp}_{C^0([0,T])}^2,
\]
which concludes the proof with $\nu= \frac{1}{4p}$.
\end{proof}

\begin{proof}[Proof of Proposition \ref{without}] We will prove that $\hp(\phi_h) \leq 1$. 
Fix $0<\eps<1$. 
As usual, we will prove that for $T$ large enough $D_T^{\phi_h}(\eps) \leq cT+d$ with $c,d$ depending only on $\eps$. 
To do this, we will construct a cover of a fundamental domain $\De$ for $\T^2$ in $\R^2$ by sets with $d_T^{\ha\phi_h}$-diameter less than $\eps$.
For $y\in \R$, we denote by $\vp_y$ the function such that $G_{\vp_y}$ is the orbit of $(0,y)$. Then $\vp_1=\vp_0+1$ and the compact domain $\De$ delimited by $\vp_0, \vp_1, V_0$ and $V_1$ is a fundamental domain for $\T^2$.

Let $T>\frac{1}{4p}$.
Since the function $y \ma A[\vp_0, \vp_y, T]$ is a homeomorphism of $\R^+$, there exists a unique $y_1>0$ such that
\[
A[\vp_0, \vp_{y_1}, T]=\min\left(c_1\eps, \demi \frac{\eps^2}{9}\right).
\]
Iterating the processus, we find an increasing sequence $(y_k)_{k\in \N}$ in $\R^+$ such that for all $k$
\[
A[\vp_{y_k}, \vp_{y_{k+1}}, T]=\min(c_1\eps, \demi \frac{\eps^2}{9}).
\]
Set $c_\eps:=\min(c_1\eps, \demi \frac{\eps^2}{9})\inv$.
Since $A[\vp_0, \vp_1,T]=T$ there exists $\ka \leq  c_\eps T+1$ such that the strips $S(\vp_k,\vp_{k+1},T)$ for $0\leq k \leq \ka$ cover $S(\vp_0,\vp_1,T)$.
Now, by the previous lemma, for any $k$, $\norm{\vp_{y_k}-\vp_{y_{k+1}}}_{C^0([0,T])}\leq \frac{\eps}{3}$.
Then by Corollary \ref{corollaire}, the intersection of any of the previous strips with $\De$ is covered by $\frac{1}{c_0\eps}+1$ subsets with $d_T^{\ha \phi_h}$-diameter less than $\eps$. 
This way we get a cover of $\De$  by $\frac{c_\eps}{c_0\eps} T + 1/(c_0\eps)+1$ subsets with $d_T^{\ha \phi_h}$-diameter less than $\eps$. The projection of these subsets on $\T^2$ yields a cover  of $\T^2$ by subsets $d_T^{\phi_h}$-diameter less than $\eps$.
Then for all $T\geq \frac{1}{4p}$, $D_T^{\phi_h}(\eps) \leq \frac{c_\eps}{c_0\eps} T + 1/(c_0\eps)+1$ and $\hp(\phi_h)\leq 1$.
\end{proof}


\subsection{Vector fields with periodic orbits}
This section is devoted to the proof of the following result.

\begin{propa}\label{periodicorbit}
Let $X$ be a $C^1$ nonvanishing vector field on $\T^2$ that possesses periodic orbits. Let $\phi_X$ be the flow of $X$. 
Then $\hp(\phi_X)\in \{0,1\}$.
Moreover $\hp(\phi_X)=0$ if and only if $\phi_X^1$ is conjugate to a rotation.
\end{propa}

To begin with, we describe briefly the dynamics of such a vector field. Since it is closely related to the dynamics of nonvanishing vector field on \emph{plane annuli}, that is, compact domains of $\R^2$ homeomorphic to $\T\times [0,1]$, we begin with the description of such systems.
The following discussion is extracted from \cite{God83} and we refer to it for a complete survey of the theory and for the proofs of the statements.
\vspace{0.2cm}

Consider a plane annulus $A$ with coordinates $(\th,r)\in \T\times [0,1]$.
Let $Y$ be a nonvanishing vector on $A$ such that $\T\times\{0\}$ and $\T\times\{1\}$ are periodic orbits for $Y$. 
We say that $A$ is a \emph{component of type} (III) if  $A$ is foliated by periodic orbits.

If there are no periodic orbits in $\stackrel{\circ}{A}$, we say that $A$ is a \emph{component of type} (II) if the orientations induced by $Y$ on the periodic orbits $\T\times\{0\}$ and $\T\times\{1\}$  do coincide with an orientation of $A$, and a \emph{component of type} (I) if not. 
Then, in both components of types (I) and (II), one of the periodic orbits $\T\times\{0\}$ and $\T\times\{1\}$ is asymptotically stable and is the $\om$-limit set of all points of $\stackrel{\circ}{A}$, and the other one is asymptotically unstable and is the $\al$-limit set of all points of $\stackrel{\circ}{A}$. The following proposition is proved in \cite{God83}.

\begin{prop}\label{conj_hyp}
We denote by $\phi=(\phi^t)_{t\in \R}$ the flow of $Y$. 
Assume that $A$ is a component of type \emph{(I)} or \emph{(II)} and that $\T\times\{1\}$ is asymptotically stable. 
There exists $\de\in\, ]0,1[$ such that
\begin{enumerate}
\item there exists $\al_+ \in \R$, $\be_+>0$ and a homeomorphism $\chi_+ :\T\times [1-\de, 1]\rit \T\times [1-\de, 1]$ such that $\chi_+ \circ \phi^t= \psi_+^t\circ \chi_+$ where $\psi_+^t: (\th,r)\ma(\th +t\al_+, re^{-t\be_+})$,
\item there exists $\al_- \in \R$, $\be_->0$ and a homeomorphism $\chi_- :\T\times [0,\de]\rit \T\times [0,\de]$ such that for all $t\geq 0$, the following diagram is commutative:

\[
  \xymatrix{
    \phi^{-t}(\T\times [0,\de]) \ar[r]^{\chi_-} \ar[d]_{\phi^t}  & \psi_-^{-t}(\T\times [0,\de])\ar[d]^{\psi_-^t} \\
    \T\times [0,\de] \ar[r]_{\chi_-} & \T\times [0,\de]
  }
\]
where $\psi_-^t: (\th,r)\ma(\th +t\al_-, re^{-t\be_-})$.
\end{enumerate}
\end{prop}

\begin{rem} \label{pointerranttype12-sectiontransverse}
1) The annulus $A$ is of type (I) if and only if $\al_+\al_->0$, and  of type (II) if and only if $\al_+\al_-<0$.

2) Any point in a component of type (II) that is not in a periodic orbit is wandering.

3) When the flow is of type (I) or (III), it admits a global transverse section that joins the boundaries $\T\times\{0\}$ and $\T\times \{1\}$. 
We can moreover chose such a section to be a $C^1$ submanifold.
This does not hold for flows of type (II), even if there still exist $C^1$ global transverse sections.
\end{rem}

Consider now a nonvanishing $C^1$ vector field $X$ on $\T^2$ that possesses periodic orbits. Assume that there are at least two periodic orbits. Then all periodic orbits are homotopic and $\T^2$ is a finite or countable union of domains $D$ bounded by two periodic orbits such that the flow $\phi_X$ of $X$ is conjugate to the flow of a component of type (I), (II) or (III). 
We say that such a domain $D$ is a \emph{zone of type} (I), (II) or (III).
If there is only one periodic orbit $\ga$, then the flow on $\T^2\setm \ga$ is conjugate to the flow of component of type (I) on the open annulus $\T\times ]0,1[$.
Conversely, if all the orbits are periodic, the flow is conjugate to the flow of a component of type (III) in a annulus $\T\times [0,1]$ with identification of the two boundaries. 
\vspace{0.2cm}

Before proving Proposition \ref{periodicorbit}, we compute the polynomial entropy for systems in components of type (I), (II) or (III).

\begin{prop}\label{type3}
Let $A$ be plane annulus and $Y$ be a vector field on $A$ such that $A$ is a component of type \emph{(III)}. Let $\phi$ be the flow of $Y$. Then $\hp(\phi)\in \{0,1\}$ and $\hp(\phi)=0$ if and only if $\phi$ is conjugate to a rotation.
 \end{prop}
 
The proof of Proposition \ref{type3} is based on the following result on the polynomial entropy for action-angle systems proved in \cite{Mar-09} in a more general case.

\begin{prop}\label{AA}
Let  $\om : [0,1]\rit \R$ be a $C^1$ function and let $\psi$ be the flow on $\T\times [0,1]$ defined by $\psi^t(\th,r)=(\th+t\om(r),r)$. 
Then $\hp(\phi):=\max_{[0,1]}\rk d\om$.
\end{prop}

\begin{proof}[Proof of Proposition \ref{type3}]
We will prove that the flow $\phi$ is $C^0$-conjugate to a flow $\psi$ of the form given in Proposition \ref{AA}. To do this we will use the existence of a $C^1$ global transverse section that joins the two boundaries of $\T\times [0,1]$.
Let $\Sig$ be such a section. The curve $\Sig$ is a graph over $\{0\}\times [0,1]$ so it is parametrized by $r\in [0,1]$. 
We set $\Sig:=\{\sig(r)\,|\, r\in [0,1]\}$.
Since the orbits are periodic the Poincaré return map of $\Sig$ is the identity. We denote by $\tau$ the return time of $\Sig$, that is, for $r\in [0,1]$, $\sig(r)=\phi_X^{\tau(r)}(\sig(r))$.
The function $\tau$ is $C^1$ and never vanishes, so the function $\om :r \ma \frac{1}{\tau(r)}$ is well defined and $C^1$.
Let $\psi:=(\psi^t)_{t\in \R}$ be the flow on $\T\times [0,1]$ defined by $\psi^t(\th,r)=(\th+t\om(r),r)$.
To construct the conjugacy $\chi$ between $\phi$ and $\psi$, we first observe that for any $z\in \T\times [0,1]\setm \Sig$, there exists a unique $t_z>0$ such that $\phi^{-t_z}(z)\in \Sig$ and $\phi^{-t}(z)\notin \Sig$ when $0\leq t < t_z$.
We define $\chi$ in the following way:
\begin{itemize}
\item for $r\in [0,1]$, $\chi(\sig(r))=(0,r)$
\item for $z\in \T\times [0,1]\setm \Sig$, $\chi(z)=\psi^{t_z}\circ \chi \circ \phi^{-t_z}(z)$.
\end{itemize} 
Let $t\in \R$ and $z\in \T\times [0,1]$. Let $r$ be such that $\phi^{-t_z}(z)=\sig(r)$. There exists a unique $m\in \Z$ and a unique $s\in [0, \tau(r)[$ such that $t+t_z=m\tau(r)+s$.
Observe that $t_{\phi_X^t(z)}=s$ and that $\phi_X^{t-s}(z)=\sig(r)$.
Then 
\[
\chi\circ \phi^t(z) =\psi^s\circ \chi(\sig(r)) = (s\om(r),r).
\]
Now $s\om(r)\equiv s\om(r)+m=(s+m\tau(r))\om(r)$ and
\begin{align*}
 ((s + m\tau(r))\om(r),r) &= ( (t+t_z)\om(r),r)\\
&  = \psi^t\circ \psi^{t_z}\circ \chi(\sig(r))\\ 
 & = \psi^t\circ \psi^{t_z}\circ\chi\circ \phi^{-t_z}(z)= \psi^t\circ \chi(z).
\end{align*}
This proves that $\chi$ conjugates $\psi$ and $\phi$.
Then $\hp(\phi)\in\{0,1\}$  and $\hp(\phi)=0$ if and only if $\om$ is constant, that is, $\phi$ is conjugate to a rotation.
\end{proof}

\begin{prop}\label{type12}
Let $A$ be plane annulus and $Y$ be a vector field on $A$ such that $A$ is a component of type \emph{(I)} or \emph{(II)}. Let $\phi$ be the flow of $Y$. Then $\hp(\phi)=1$.
\end{prop}

\begin{proof}
We first observe that since $\phi$ possesses wandering points, $\hp(\phi)\geq 1$. Let us prove that $\hp(\phi)\leq 1$. As usual, for $\eps>0$ fixed and $n$ large enough, we construct a cover of $A$ by sets with $d_n^\phi$-diameter less than $\eps$ with cardinal of the form $cn+d$ with $c,d$ depending only on $\eps$. The idea is essentially the same as in the proof of lemma \ref{prop:indseg}.

We set $\ga_-:=\T\times \{0\}$, $\ga_+:=\T\times \{1\}$ and we assume that $\ga_+$ is asymptotically stable.
Let $\de>0$ be given by Proposition \ref{conj_hyp} and fix $\eps\in\, ]0, \de[$. 
We set $\CC_+:=\T\times [1-\eps,1], S_+:=\T\times\{1-\eps\}, \CC_-:=\T\times [0, \eps]$ and $S_-:=\T\times \{\eps\}$. 
Observe that $S_-$ and $S_+$ are global transverse sections of $\phi$.
Finally we set $\II:= \ov{A\setm \left (\CC_- \cup \CC_+\right)}$.
\vspace{0.2cm}

\noindent $\bullet$ \textit{Cover of $\CC_+$:} Let $\chi_+$ and $(\psi_+^t)$ be such as in Proposition \ref{conj_hyp} 1). 
Observe that for all $t>0$ and all $(z,z')\in \CC_+^2$, $d(\psi^t(z),\psi^t(z'))\leq d(z,z')$.
Let $\eta>0$ such that for all $(z,z')\in \CC_+^2$, if $d(z,z')\leq \eta$ then $d(\chi_+\inv(z),\chi_+\inv(z))\leq \eps$, and let $\eta'$ such that if $d(z,z')\leq \eta'$ then $d(\chi_+(z),\chi_+(z))\leq \eta$.
Let $\eta_0=\min (\eta, \eta',\eps)$.

Since for all $t>0$, $d(\phi^t(z),\phi^t(z'))=d(\chi_+\inv\circ \psi\circ \chi_+(z),\chi_+\inv\circ \psi\circ \chi_+(z))$, any subset of $\CC_+$ with diameter less than $\eta_0$ has $d_n^\phi$-diameter less than $\eps$ for any $n\in \N$. 
By compactness of $\CC_+$, we can cover $\CC_+$ by a finite number $m$ (independent of $n$) of balls of radius less than $\demi\eta_0$. 
\vspace{0.2cm}

\noindent $\bullet$ \textit{Cover of $\II$:} Let  $\wp$ be the Poincaré map between $S_-$ and $S_+$ and let $\tau$ be its time function, that is, for $z\in S^-$, $z\in S_-$, $\wp(z)=\phi^{\tau(z)}(z)$.  
The function, $\tau$ is continuous on $S^-$.  
We set $\tau_m:=\max\tau$. 
Then for all $n>\tau_m$ and for all $z\in \II$, $\phi^n(z)\in \CC_+$. 
Fix $n_0\geq \tau_m$. 
There exists $\eta_1\in\, ]0,\eta_0]$ such that for all $(z,z')\in \II^2$ and all $n\in \{0,\dots, n_0\}$, if $d(z,z')\leq \eta_1$, then $d(\phi^n(z),\phi^n(z'))\leq \eta_0$.
Then, any subset of $\II$ with diameter less than $\eta_1$ has $d_N^\phi$-diameter less than $\eps$ for all $N\in \N$.
Again, since $\II$ is compact, we can cover it by a finite number $p$ (independent of $N$ in $\N$) of balls of radius less than $\demi\eta_1$.
\vspace{0.2cm}

\noindent $\bullet$ \textit{Cover of $\CC_-$:} Fix $n\in \N$. Let $B_1,\dots, B_p$ the balls of the previous cover of $ \II$. 
Let $D$ be the compact domain contained in $\II$ delimited by $S_-$ and $\phi(S_-)$. Then $\phi\inv(D)$ is covered by a finite number $r_1\leq p$ of domains of the form $\phi\inv(B_j)$, for $1\leq j \leq p$. Such a domain can be cover by a finite number of balls with diameter $\eps$. Let $s_1$ be the maximal number of such a covering. Then we get a cover of $\phi\inv(D)$ by at most $q_1=s_1r_1$ balls with $d_N^\phi$-diameter less than $\eps$ for all $N\in \N$.
Considering the inverse images of these balls by $\phi$, and covering again each of these sets by a finite number of balls with diameter $\eps$, we get a cover of $\phi^{-2}(D)$ by a finite number $q_2$ (independent of $N$ in $\N$) of balls with $d_N^\phi$-diameter less than $\eps$ for all $N\in \N$.
Iterating the processus, for all $1\leq k \leq n$, we obtain a cover  of $\phi^{-k}(D)$ with a finite number $q_k$ (depending only on $\eps$) of balls   with $d_N^\phi$-diameter less than $\eps$ for all $N\in \N$. Let $q=\max\{q_1,\dots,q_n\}$. We have got a cover of the domain bounded by $S_-$ and $\phi^{-n}(S_-)$ with at most $nq$ balls with $d_N$-diameter less than $N$, for all $N\in \N$.

It remains to cover the domain $\De_-$ bounded by $\ga_-$ and $\phi^{-n}(S_-)$. 
Let $(\psi_-^t)$ and $\chi_-$ such as in Proposition \ref{conj_hyp} 2).
Let $z\in \De_-$. Then for all $1\leq k \leq n$, $\phi^k(z)\in \CC_-$ and $\chi\circ \phi^k(z)=\psi^k\circ \chi(z)$.
So $\psi^k(\chi(\De_-))\subset \CC_-$. 
Therefore $\chi(\De_-)\subset \T\times [0, \eps e^{-n\be_-}]$.
Let $\eta>0$ such that if $d(z,z')\leq \eta$, then $d(\chi_-\inv(z), \chi_-\inv(z'))\leq \eps$.

Let $I_1, \dots, I_\ell$ be a finite cover of $\T$ by compact intervals of length less than $\eta$.
Let $\ka \in \N$  with $\ka \leq \frac{\eps}{\eta}+1$ and $0=r_0<r_1\dots <r_\ka =\eps$ such that 
\[
r_{i+1}-r_i \leq \eta \quad {\rm{and}}\quad \bigcup_{1=i}^{\ka-1} [r_i, r_{i+1}]=[0, \eps].
\]
Let $r_i'=r_i e^{-n\be_-}$. Then 
\[
r_{i+1}'-r_i' \leq \eta e^{-n\be_-} \quad {\rm{and}}\quad \bigcup_{1=i}^{\ka-1} [r_i', r_{i+1}']=[0, \eps e^{-n\be_-}].
\]
For $(j,i)\in \{1,\dots, \ell\}\times\{1, \dots, \ka\}$, we set $D_{ji}:=I_j\times[r_i, r_{i+1}]$. 
The sets $D_{ji}$ cover $\chi(\De_-)$ and one immediately checks that each $D_{ji}$  has $d_n^\psi$-diameter less than $\eta$.
Now, each of the sets $\chi\inv(D_{ji})$ can be covered by a finite number (independent of $n$) of sets  with diameter less than $\eps$. 
By construction, such subsets cover $\De_-$ and have $d_n^\phi$-diameter less than $\eps$. The cardinal $r$ of this cover is independent of $n$.
\vspace{0.2cm}

Finally, we get a cover of $A$ with $nq +m +p +r$ sets of $d_n^\phi$-diameter less than $\eps$, where $q,m,p$ and $r$ only depend on $\eps$. This proves that $\hp(\phi)\leq 1$.
\end{proof}

\begin{rem}\label{nombrefini}
Observe that using proprery \ref{ptehphw} 5) of $\hp$, Proposition \ref{periodicorbit} is proved in the case when  $\T^2$ is covered by a \emph{finite} union of zones of type (I), (II) or (III).
\end{rem}

The proof in the general case is based on the three following lemmas. The first one is a classical result, whose proof can be found in \cite{PdM82} for instance.

\begin{lem}\label{sectiontransverse}
Each periodic orbit $\Ga$ admits a transverse section $\Sig$, with $\{a\}:=\Ga\cap\Sig$ such that there exists an open neighborhood $O\subset \Sig$ of $a$, such that the Poincaré return map $\wp : O\rit \Sig$ is well defined.
\end{lem}

\begin{lem}\label{hpdansS}
Assume that $\wp$ has a fixed point $b\in O$ with $b\neq a$. Let $[a,b]_\Sig$ be the compact segment of $O$ limited by $a$ and $b$. 
Set
$
 S:=\bigcup_{t\in \R}\phi_X^t([a,b]_{\Sig}).
$
Then $\hp(\phi_X, S)\leq 1$.
\end{lem}

\begin{proof}
Observe that the orbit of $b$ is periodic.
Let $\tau$ be the transition time of $\wp$, that is, for $z\in O$, $\wp(z)=\phi^{\tau(z)}(z)$. For any $z\in [a,b]_\Sig$, $\tau(z)>0$.
Let $\xi : [0,1]\to [a,b]_\Sig$ be a $C^1$ diffeomorphism with $\xi(0)=a$ and $\xi(1)=b$.
Let $\wp^*=\xi\inv\circ\wp\circ\xi:[0,1]\to[0,1]$, and let $\bar{\tau}:=\tau\circ \xi$.
We first note that the construction of a suspension in Lemma \ref{suspension} is still valid for a homeomorphism $f: [0,1]\rit [0,1]$.
Consider  a vector field $\ha{X}_\wp$ on $\R\times [0,1]$ with flow $\ha{\phi}_\wp$ such that 
\begin{itemize}
\item $\ha{X}_\wp(x_1,x_2)=(1, X_2(x_1,x_2))$ with $X_2(x_1+1,x_2)=X_2(x_1,x_2)$ for all $(x_1,x_2)\in \R\times [0,1]$. 
\item $\ha\phi_\wp^1(n,x_2)=(n+1, \wp(x_2))$ for all $(n,x_2)\in \Z\times [0,1]$.
\end{itemize}
The projection $X_\wp$ of $\ha{X}_\wp$ on $\T\times [0,1]$ is a suspension of $\wp^*$. We denote by $\phi_\wp$   its flow.
The orbits of $\ha{X}_\wp$ are graph over the $x_1$-axis, so Proposition \ref{L1Deviation} and Corollary \ref{corollaire} hold true.
For $x_1\in [0,1]$, we denote by $\vp_{x_1}$ the orbit of $(0,{x_1})$. 
Obviously, $\vp_0=\R\times\{0\}$ and $\vp_1=\R\times \{1\}$.
Therefore we can apply Lemma \ref{minorationaire} and the same proof than for Theorem \ref{without} yields $\hp(\phi_\wp) \leq 1$.

We will now show that the restriction of $X$ to $S$ is conjugate to $X_\wp$, which will conclude the proof.
The curve $\Sig^*:=\{0\}\times [0,1]$ is a transverse section of $X_\wp$ with Poincaré return map $P:(0,x)\ma (0, \wp^*(x))$. Then $P(0,x)=\phi_\wp^{\bar{\tau}(x)}(0,x)$.
For $z\in S\setm \Sig$, we set $t_z:=\inf\{t>0\,|\, \phi_X^{-t}(z)\in \Sig\}$.
We define a map $\chi : S\rit \T\times [0,1]$ in the following way:
\begin{itemize}
\item for $z\in \Sig$, $\chi(z)=(0,\xi\inv(z))$
\item for $z\in S\setm \Sig$, $\chi(z)=\phi_\wp^{t_z}\circ \chi \circ \phi_X^{-t_z}(z)$.
\end{itemize} 
Observe that for $z\in \Sig$ and $n\in \Z$, $\chi\circ \wp^n(z)=P^n\circ \chi(z)$.
Moreover if $\tau_m$ is such that $\phi_X^{\tau_m}(z)= \wp^m(z)$, then $P^m(\chi(z))= \phi_\wp^{\tau_m}(\chi(z))$.
Let $z\in S$ and $t\in \R$. Let $s=t_{\phi_X^t(z)}$ and $z_1:= \phi_X^{-s}(\phi_X^t(x))\in S$. Let $z_0=\phi_X^{-t_z}(z)$.
There exists a unique $m\in \Z$ such that $z_1=\wp^m(z_0)$. Then $\phi_X^{t}(z)=\phi_X^{s}(\wp^m(z_0))$.
Therefore
\begin{align*}
\chi\circ \phi_X^t(z) &= \chi \circ \phi_X^s(\wp^m(z_0)) = \phi_\wp^s \circ \chi \circ\phi_X^{-s}(\phi_X^s(\wp^m(z_0)))\\
 &=\phi_\wp^s\circ \chi(\wp^m(z_0))=\phi_\wp^s \circ P^m(\chi(z_0)).
\end{align*}
Let $\tau_m$ be such that $\wp^m(z_0)=\phi_X^{\tau_m}(z_0)$. Then $\phi_X^t(z)=\phi_X^{s+\tau_m}(z_0)$, so $t=s+\tau_m-t_z$. 
Therefore
\[
\phi_\wp^s \circ P^m(\chi(z_0))= \phi_\wp^{s+\tau_m}(\chi(z_0))=\phi_\wp^{t+t_z}(\chi(z_0))=\phi_\wp^t(\chi(z))
\]
and $\chi$ is a conjugacry between $\phi_X$ restricted to $S$ and $\phi_\wp$. This proves that $\hp(\phi_X,S)\leq 1$.
\end{proof}

\begin{lem} 
Any periodic orbit $\Ga$ of $X$ admits a compact neighborhood $\NN_\Ga$ in $\T^2$, $\phi_X$-invariant  such that 
$
\hp(\phi_X,\NN_\Ga)\leq 1.
$
\end{lem}

\begin{proof} We denote by $\PP$ the set of periodic orbits. For $\Ga \in \PP$  the two following cases only occur.
\begin{enumerate}
\item $\Ga$ is in the interior of a zone of type (III)
\item $\Ga$ is the common boundary of two zones of type (I), (II) or (III), with at most one of each of type (III). We denote by $\ha{\PP}$ the subset of these orbits. It is a finite or countable subset.
\end{enumerate}
In the first case, the lemma is proved by Proposition \ref{type3}: one just has to choose for $\NN_\Ga$ the zone of type (III) that contains $\Ga$.
When $\ha{\PP}$ is finite, the lemma is immediate by Propositions \ref{type3} and \ref{type12} (see remark \ref{nombrefini}) : for $\Ga \in \ha{\PP}$ one can choose for $\NN_\Ga$ the union of the two zones of type (I), (II) or (III) with common boundary $\Ga$.
The only difficulty occurs when $\ha{\PP}$ is infinite and when $\Ga$ is an \emph{accumulation point in $\ha{\PP}$}, that is, there exists a sequence  $(\Ga_n)_{n\in \N} \in \ha{\PP}^\N$ such that  $d(\Ga, \Ga_n) \rit_{n\rit \infty}=0$ where $d$ is the Hausdorf distance between two compacts subsets of $\T^2$ induced by the classical Euclidean distance of $\R^2$. 
\vspace{0.2cm}

Assume that $\ha{\PP}$ is infinite and let $\ha\Pi$ be the union of the periodic orbits in $\ha{\PP}$.
Let $\Ga$ be an accumulation point in $\ha{\PP}$. 
Fix $a\in \Ga$ and let $\Sig$ be a curve transverse to $\Ga$ in $a$ such as in lemma \ref{sectiontransverse}. 
Then at least one of the connected component of $\Sig\setm\{a\}$ has an infinite intersection with $\ha\Pi$. We distinguish the cases when both connected component have an infinite intersection with $\Pi$  and when only one has an infinite intersection with $\ha\Pi$.
\vspace{0.2cm}

\noindent $\bullet$ Both connected components of $\Sig \setm \{a\}$ have an infinite intersection with $\ha\Pi$. 
We denote by $\Sig_1$ and $\Sig_2$ the connected components of $\Sig \setm \{a\}$. 
We set $\ha\Pi\cap \Sig_1=\{b_n\,|\,n\in \N\}$ and $\ha\Pi\cap \Sig_2=\{c_n\,|\,n\in \N\}$ such that $\lim_{n\rit \infty}b_n=a=\lim_{n\rit \infty} c_n$. 
We moreover assume that the sequence $|\xi\inv(b_n)-\xi\inv(a)|$ and $|\xi\inv(c_n)-\xi\inv(a)|$ are decreasing (with limit $0$). 
Let $[b_0,c_0]_\Sig$ be the compact segment of $\Sig$ bounded by $b_0$ and $c_0$ that contains $a$.
Set $\NN_\Ga:= \bigcup_{t\in \R}\left(\phi^t_X([b_0,c_0]_\Sig\right)$. Then $\NN_\Ga$ is a compact neighborhood of $\Ga$, $\phi_X$-invariant and by lemma \ref{hpdansS}, $\hp(\phi_X, \NN_\Ga)\leq 1$.
 
\noindent $\bullet$ Only one connected component of $\Sig \setm \{a\}$ have an infinite intersection with $\ha\Pi$. Let $\Sig_f$ be the connected component with finite (and possibly empty) intersection with $\ha\Pi$ and $\Sig_i$ be the other one. 
As before, we set $\ha\Pi\cap \Sig_i:=\{c_n\,|\,n \in \N\}$ with $\lim_{n\rit \infty}c_n=a$ and $|\xi\inv(c_n)-\xi\inv(a)|$ decreasing. Let $[a,c_0]_\Sig$ be the compact segment of $\Sig$ bounded by $a$ and $c_0$ and set $S_i:=\bigcup_{t\in \R}\left(\phi^t_X([a,c_0]_\Sig\right)$. Then $\hp(\phi_X, S_i) \leq 1$.

If $\Sig_f\cap \ha\Pi =\emptyset$, then $\Sig_f$ is contained in the interior of a zone  $D$ of type (I), (II) or (III) and  $\hp(\phi_X,D)\leq 1$.
Then the union $\NN_\Ga:= S_i\cup D$ is a compact and $\phi_X$-invariant neighborhood of $\Ga$ and using property \ref{ptehphw} 5), one sees that $\hp(\phi_X, \NN_\Ga)\leq 1$.

If $\Sig_f\cap \Pi \neq\emptyset$, we choose $b\in \Sig_f\cap \Pi$, we denote by $[a,b]_\Sig$ the compact segment of $\Sig$ bounded by $a$ and $b$ and we set $S_f:=\bigcup_{t\in \R}\left(\phi^t_X([a,b]_\Sig\right)$.
Again, one immediately sees that $\NN_\Ga:= S_i\cup S_f$ satisfies all the required properties.
\end{proof}

\begin{proof}[Proof of Proposition \ref{periodicorbit}] 
Observe first that the union $\Pi$ of all periodic orbits is a compact subset of $\T^2$. Indeed, its complementary is open since it is the union of the interiors of zones of type (I) or (II).
Therefore $\Pi$ admits a finite covering by invariant subsets over which $\hp(\phi_X) \leq 1$. Now the complementary of this covering is a finite union of interior of zones of type (I), (II) or (III).
This way, we get a finite covering of $\T^2$ by $\phi_X$-invariant subset on which $\hp \leq 1$. By Proposition \ref{type12}, one sees that as soon as there exists a zone of type (I) or (II), $\hp(\phi_X)=1$. Then, according to Proposition \ref{type3} $\hp(\phi_X)=0$ if and only if $\phi_X$ is conjugate to a rotation.
\end{proof}


\end{document}